\documentclass[11pt,a4paper]{article}
\usepackage[T1]{fontenc}
\usepackage{pgf}
\usepackage{tikz}
\usepackage[symbol]{footmisc}
\usetikzlibrary{arrows,automata}
\usepackage{verbatim}
\usepackage{caption}
\usepackage{epsf,epsfig,amsfonts,amsgen,amsmath,amstext,amsbsy,amsopn,amsthm,amssymb}
\usepackage{color,graphicx,xcolor}
\usepackage{mathrsfs,hyperref}
\usepackage{enumerate}
\usepackage{url}
\usepackage[shortlabels]{enumitem}
\usetikzlibrary{patterns}
\usepackage{algorithm, algorithmic}
\usepackage{longtable,multirow,array,booktabs}
\usepackage{changepage}

\usepackage{setspace}
\newtheorem{dfn}{Definition}[section]
\newtheorem{thm}[dfn]{Theorem}

\newtheorem{lem}[dfn]{Lemma}
\newtheorem{fact}[dfn]{Fact}
\newtheorem{obs}[dfn]{Observation}

\newcommand{\dB}{{\mathcal{B}}}

\def\wsat{\mathrm{wsat}}
\makeatletter

\newcommand{\Rmnum}[1]{\expandafter\@slowromancap\romannumeral #1@}
\makeatother

\DeclareMathOperator{\sd}{sd}

\title{The Complexity of Weak Saturation for Complete Graphs and Balanced Complete Bipartite Graphs}
\author{Yihan Chen\footnote{School of Mathematical Sciences, University of Science and Technology of China, Hefei, Anhui 230026, China (\texttt{cyh2020@mail.ustc.edu.cn})}~~~~~~ Tianying Xie\footnote{School of Mathematics and Statistics, Fuzhou University, Fuzhou, Fujian 350108, China (\texttt{xiety@fzu.edu.cn})}~~~~~~}
\date{\today}

\begin{document}

\maketitle

\begin{abstract}
For graphs $F$ and $H$, a spanning subgraph $G$ of $F$ is weakly $H$-saturated in $F$ if the edges in $E(F)\setminus E(G)$ can be added one at a time, each addition creating a new copy of $H$. Recently, Tancer and Tyomkyn proved that, given an $n$-vertex graph $F$, deciding whether $\wsat(F,K_3)=n-1$ is NP-hard. In this paper, we study the decision version of the weak saturation problem and show that, for every fixed integer $r\ge 3$, given a graph $F$ and an integer $k$, deciding whether $\wsat(F,H)\le k$ is NP-complete when $H\in\{K_r,K_{r,r}\}$. Our approach uses novel graph-theoretic and topological ideas and techniques, yielding new constructions that build on the construction of Tancer and Tyomkyn. In particular, our proofs bring the flag-no-square property, a fundamental property in topology that is of independent interest, into the study of weak saturation problem.

\end{abstract}

\section{Introduction}\label{Sec:Introduction}
Let $F$ and $H$ be finite graphs.
A spanning subgraph $G\subseteq F$ is called weakly $H$-saturated in $F$ if the missing edges $E(F)\setminus E(G)$ can be ordered as $e_1,\ldots,e_m$ so that, whenever $e_i$ is added to $G+\{e_1, \cdots, e_{i-1}\}$, the new edge $e_i$ lies in a copy of $H$ in the resulting graph.
The \emph{weak saturation number} is defined as the following,
\[
    \wsat(F,H)=
    \min\{|E(G)|:G\subseteq F\text{ is weakly $H$-saturated}\}.
\]
In the special case $F=K_n$, we write $\wsat(n,H)$ for $\wsat(K_n,H)$.

The study of weak saturation problem was initiated by Bollob\'as~\cite{Bol68} in 1968, motivated by saturation problems for $k$-uniform hypergraphs. He determined $\wsat(n,K_r)$ for $3\le r< 7$ and conjectured that $\wsat(n,K_r)=\binom{n}{2}-\binom{n-r+2}{2}$ for $n\ge r\ge 2$.
This conjecture was later proved by Lov\'asz~\cite{Lov77}.
Independent algebraic proofs were subsequently given by Frankl~\cite{Fra82}, Kalai~\cite{Kal84} and Alon~\cite{Alo85}.

Explicit constructions often provide upper bounds for weak saturation numbers, but determining the exact value usually requires sharp lower bounds, which are often the main difficulty.
Kalai~\cite{Kal84} introduced a linear-algebraic method for proving lower
bounds, while Alon~\cite{Alo85} proved that for every fixed graph $H$ there
exists a constant $c_H$ such that $\wsat(n,H)=c_H n+o(n)$.
More recently, Terekhov and Zhukovskii~\cite{TZ23,TZ24} developed new
combinatorial methods for lower bounds and constructed examples showing
limitations of the classical linear-algebraic approach. On the other hand,
understanding the possible values of the limiting coefficient $c_H$ has also
received considerable attention. Ascoli and He~\cite{AH25} characterized all
possible rational values of $c_H$.

Weak saturation numbers have also been investigated for a wide range of graph and hypergraph families, including sparse graphs, multiple copies of fixed graphs, complete bipartite patterns, product graphs, and bipartite or multipartite hypergraphs; see, e.g.,~\cite{AVZ25,BBMR12,BS02,BTT23,CFFS21,FG14,FGJ13,KMM21,MMT24,MNS17,MS15,ShT23}

Recently, the computational complexity of weak saturation number has also attracted attention. Tancer and Tyomkyn~\cite{TT25} proved that given a graph $F$ with $n$ vertices as input, it is NP-hard to decide whether $\wsat(F, K_3) = n - 1$. Their proof connects $K_3$ weak saturation to shellability and collapsibility of $2$-dimensional complexes. Before stating their result, we introduce the terminologies needed below. A pure complex is \emph{shellable} if its maximal faces admit an ordering $F_1,\ldots,F_m$ such that, for every $i>1$,
$$
F_i\cap \left(\bigcup_{j<i}F_j\right)
$$
is a nonempty pure complex of dimension $\dim F_i-1$. Such an ordering is called a \emph{shelling order}. A face $\tau$ is a \emph{free face} if it is properly contained in a unique maximal face $\sigma$. A \emph{collapse} for free face $\tau$ and maximal face $\sigma$ removes all faces $\gamma$ with $\tau\subseteq \gamma\subseteq\sigma$; when $\dim\tau=\dim\sigma-1$, it is called an \emph{elementary collapse}. Note that a collapse can be decomposed into a sequence of elementary collapses. A complex is \emph{collapsible} if a finite sequence of collapses reduces it to a single vertex. With these terminologies in place, we state Theorem~3 of~\cite{TT25}, the key result underlying their reduction.

\begin{thm}[Tancer and Tyomkyn~\cite{TT25}]\label{thm:TT-ori}
There is a polynomial time algorithm that produces from a given
$3$-CNF formula $\phi$ with $t$ variables a pure $2$-dimensional connected
complex $L_\phi$ with $\widetilde{\chi}(L_\phi)=t$ such that the following
statements are equivalent:
\begin{enumerate}
    \item[(i)] \textit{The formula $\phi$ is satisfiable.}
    \item[(ii)] \textit{The complex $L_\phi$ is shellable.}
    \item[(iii)] \textit{The complex $L_\phi$ is collapsible after removing some $t$ triangles.}
    \item[(iv)] \textit{We have $\operatorname{wsat}(L_\phi^{(1)},K_3)=n-1$ where $n$ is the number of vertices of $L_\phi$.}
    \item[(v)] \textit{The complex $L_\phi$ is collapsible after removing some number of triangles.}
\end{enumerate}
\end{thm}

In this paper, we extend Tancer and Tyomkyn's result from $K_3$ to all complete graphs $K_r$ and balanced complete bipartite graphs $K_{r,r}$ with $r\ge 3$, by giving polynomial-time reductions from their NP-hard problem of deciding, input an $n$-vertex graph $F$, whether $\wsat(F,K_3)=n-1$, to the corresponding decision problems for $K_r$ and $K_{r,r}$. In particular, our proofs bring the \emph{flag-no-square} property (also known as \emph{$5$-large}), a widely studied property in topology that is of independent interest~\cite{DLMR26,Dra99,JS06,KNS25,KW16,KPP12,PS09}, into the study of weak saturation problems. This property also has many applications in other areas such as geometric group theory and combinatorial commutative algebra~\cite{AOS24,CDSS25,CKV16,CKV19,Gro87,Osa13}. More precisely, for a fixed graph $H$, we consider the following decision problem: given a finite graph $F$ and an integer $k$, decide whether $\wsat(F,H)\le k$. The following is our main result.

\begin{thm}\label{thm:intro-cliques}
For every fixed integer $r\ge 3$ and every $H\in\{K_r,K_{r,r}\}$, given a finite graph $F$ and an integer $k$ as input, it is NP-complete to decide whether $\wsat(F,H)\le k$.
\end{thm}

The paper is organized as follows. Section~\ref{sec:prelim} introduces notations, definitions and some facts; Section~\ref{sec:cliques} introduces a lifting lemma and proves the result for complete graphs; Section~\ref{sec:balanced} proves the result for $K_{3,3}$ first and then lifts it to balanced bipartite graphs $K_{r,r}$ with $r\ge 4$; Section~\ref{sec:concluding} discusses the limits of our method in the case of $C_4$ and the case of $K_{s,t}$ where $s\neq t$.

\section{Preliminaries}\label{sec:prelim}

In this section we introduce the notations, definitions and some basic facts. We first introduce notations in graph theory. All graphs in this paper are finite and simple.
For a graph $G$, let $V(G)$ and $E(G)$ denote its vertex set and edge set, let $\omega(G)$ denote its clique number. For any vertex $v\in V(G)$, denote the neighborhood of $v$ in $G$ by $N_{G}(v)$. 
For two graphs $G$ and $H$, we say $G$ is $H$-free if $G$ contains no copy of $H$.
For two graphs $G_1$ and $G_2$, let $G_1\nabla G_2$ denote the graph join of $G_1$ and $G_2$. Namely, $G_1 \nabla G_2$ is obtained from the disjoint union of $G_1$ and $G_2$ by joining each vertex of $G_1$ to each vertex of $G_2$. 

Next, we recall the topological terminology needed in this paper. A \emph{simplicial complex} $C$ is a collection of finite vertex sets that is closed under taking subsets. Its elements are called \emph{faces} or \emph{simplices}; a face $\sigma$ has dimension $\dim\sigma=|\sigma|-1$, and the dimension of $C$ is the maximum dimension of its faces. A maximal face under inclusion is called a \emph{facet}, and $C$ is \emph{pure} if all facets have the same dimension. For a simplicial complex $C$, we denote $\chi(C)$ its Euler characteristic.

In this paper, we only consider finite $2$-dimensional simplicial complexes.
For such a complex $C$, let $C^{(k)}$ be the \emph{$k$-skeleton} of $C$, which is the subcomplex consisting of all faces of $C$ of dimension at most $k$. Note that $C^{(1)}$ is a graph.
A \emph{subdivision} of $C$ is a simplicial complex with the same underlying polyhedron whose simplices subdivide the simplices of $C$. 
For a finite pure $2$-dimensional simplicial complex $C$, the boundary $\partial C$ is the subcomplex generated by all edges of $C$ that are contained in exactly one $2$ dimensional face of $C$. 

\subsection{Flag complex and collapsibility}
We first introduce flag complex and barycentric subdivision.
\begin{dfn}[Flag complex]
A simplicial complex $L$ is called \emph{flag} if every clique in its $1$-skeleton $L^{(1)}$ spans a simplex of $L$.
\end{dfn}
The following is a simple observation which will be used frequently.
\begin{obs}\label{obs:flag-k4free}
If $L$ is a $2$-dimensional flag complex, then $L^{(1)}$ is $K_4$-free.
\end{obs}
\begin{proof}
Since $L$ is a flag complex, four pairwise adjacent vertices in $L^{(1)}$ would span a $3$-simplex in $L$, contradicting $\dim L\le 2$.
\end{proof}

\begin{dfn}[Barycentric subdivision]
Let $X$ be an (abstract) simplicial complex. The \emph{barycentric subdivision} $\sd(X)$ of $X$ is an (abstract)
simplicial complex defined as follows. The vertex set of $\sd(X)$ is the set of simplices $S(X)$ of $X$. A collection of simplices $\{\sigma_0,\sigma_1,...,\sigma_p\}$ of $X$ form a $p$-simplex of $\sd(X)$ precisely when they form a chain under inclusion.
\end{dfn}
It is easy to see that the barycentric subdivision $\sd C$ of every finite simplicial complex $C$ is flag. Indeed, the vertices of $\sd C$ are the simplices of $C$, and adjacency in $(\sd C)^{(1)}$ means comparability by inclusion.
A finite set of pairwise comparable simplices forms a chain, and chains precisely give the simplices in $\sd C$.

We next turn to the collapsibility of complexes. Recall that we have defined collapse and collapsibility in Section~\ref{Sec:Introduction}. In the following, we introduce some notation and useful facts. 
For two simplicial complexes $C_1$ and $C_2$, we write $C_1\searrow C_2$ if the simplicial complex $C_1$ collapses to the subcomplex $C_2$ by a finite sequence of collapses. By a \emph{triangulated disk} we mean a simplicial complex whose geometric realization is homeomorphic to a disk; such a triangulated disk $D$ is called \emph{endo-collapsible} if there exists a $2$-dimensional face $\sigma$ such that $(D\setminus{\sigma}) \searrow \partial D$, where only the face $\sigma$ is removed and all its proper faces are kept.

\begin{fact}[See~\cite{ST23}]\label{lem:subdivision-collapsible}
Let $C$ be a simplicial complex of dimension at most $2$ and $C^{\prime}$ be a subdivision of $C$. Then $C$ is collapsible if and only if $C^{\prime}$ is collapsible.
\end{fact}

\begin{fact}[See~\cite{AB20}]\label{lem:disk-endo-collapsible}
In dimension $2$, all disks are endo-collapsible. 
\end{fact}

\begin{fact}\label{lem:collapse-euler}
Collapse does not change the Euler characteristic. 
\end{fact}
\begin{proof}
    Recall that a collapse can be decomposed into a sequence of elementary collapses. An elementary collapse removes exactly two faces whose dimensions differ by one, and hence their contributions to the alternating sum in the definition of Euler characteristic cancel. Therefore a sequence of elementary collapses preserves the Euler characteristic, which implies the fact.
\end{proof}

With these in mind, we extract the following consequence from Theorem~\ref{thm:TT-ori} and its proof in \cite{TT25}, which will be used later.

\begin{thm}[Tancer and Tyomkyn~\cite{TT25}]\label{thm:TT}
There is a polynomial-time construction which maps any $3$-CNF formula $\varphi$ to a finite connected pure $2$-dimensional simplicial complex $L_\varphi$ with the following properties:
\begin{enumerate}[label=\textup{(\roman*)}]
\item $L_\varphi$ is the barycentric subdivision of some $2$-dimensional complex, and hence $L_\varphi$ is flag;
\item $\varphi$ is satisfiable if and only if $L_\varphi$ becomes collapsible after deleting a set of $2$-dimensional faces;
\item $\varphi$ is satisfiable if and only if $\wsat(L_\varphi^{(1)},K_3)=|V(L_\varphi)|-1.$
\end{enumerate}
Consequently, given a connected graph $F$ on $n$ vertices, deciding whether $\wsat(F,K_3)=n-1$ is NP-hard.
\end{thm}

\subsection{Weak saturation promblem and its opposite perspective}
For the weak saturation promblem, we have the following results.
\begin{obs}\label{lem:np-membership}
For every fixed graph $H$, the $H$-weak saturation decision problem belongs to NP.
\end{obs}

\begin{proof}
Let $F$ be the host graph. A certificate of the $H$-weak saturation decision problem consists of an initial edge set $E_0\subseteq E(F)$ with $|E_0|\le k$, together with an ordering of $E(F)\setminus E_0$.
Since $H$ is fixed, at each step one can check in polynomial time, by enumerating the possible copies of $H$ in $F$, whether the newly added edge lies in a new copy of $H$.
\end{proof}

\begin{lem}\label{lem:connected-lower}
Let $F$ be a  connected graph, then every weakly $K_3$-saturated subgraph of $F$ is connected.
Consequently,
$$
    \wsat(F,K_3)\ge |V(F)|-1.
$$
\end{lem}

\begin{proof}
Let $H$ be a weakly $K_3$-saturated subgraph of $F$. Then there is an ordering $e_1,e_2,\ldots,e_m$ of the edges in $E(F)\setminus E(H)$ such that, if $H_0=H$ and $H_i=H_{i-1}+e_i$ for $1\le i\le m$, then the addition of $e_i$ creates a new copy of $K_3$ in $H_i$.

We claim that the connectivity does not change during this process. Indeed, write $e_i=u_iv_i$. Since adding $e_i$ creates a new $K_3$, there exists a vertex $w_i$ such that both $u_iw_i$ and $v_iw_i$ are already edges of $H_{i-1}$. Hence $u_i$ and $v_i$ lie in the same connected component of $H_{i-1}$. Therefore $H_i$ and $H_{i-1}$ have the same number of connected components for every $i$.

It follows by induction that $H$ and $F=H_m$ have the same number of connected components. Since $F$ is connected, $H$ is connected. Therefore every weakly $K_3$-saturated subgraph of $F$ has at least $|V(F)|-1$ edges, and hence $\wsat(F,K_3)\ge |V(F)|-1.$
\end{proof}
The weak saturation problem can also be viewed from the opposite perspective.
\begin{dfn}\label{def:lambda}
For finite graphs $F$ and $H$, let $\lambda_H(F)$ be the maximum length of an edge-deletion sequence starting from $F$ in which each deleted edge belongs, at the moment of deletion, to a copy of $H$ in the current graph.
We call such a sequence an \emph{$H$-edge peeling sequence}.
\end{dfn}

\begin{lem}\label{lem:reverse-deletion}
For all finite graphs $F$ and $H$,
$$
    \wsat(F,H)=|E(F)|-\lambda_H(F).
$$
\end{lem}

\begin{proof}
Let $G\subseteq F$ be weakly $H$-saturated with edge addition order $e_1,\ldots,e_m$, then deleting $e_m,\ldots,e_1$ from $F$ gives an $H$-edge peeling sequence of length $m$.
Conversely, reversing any $H$-edge peeling sequence gives a valid weak saturation addition order.
Thus maximizing the number of deleted edges is equivalent to minimizing the number of initial edges.
\end{proof}

\section{Complete graphs}\label{sec:cliques}

Before proving Theorem~\ref{thm:intro-cliques} for complete graphs, we will show a lifting lemma first.
\begin{lem}\label{lem:clique-cone}
Let $q\ge 3$, and let $X$ be a finite graph with $\omega(X)\le q$.
Let $\widehat X=X\nabla K_1,$
and denote the new vertex in $V(\widehat X)\setminus V(X)$ by $a$.
Then
\[
    \wsat(\widehat X,K_{q+1})=\wsat(X,K_q)+|V(X)|.
\]
\end{lem}

\begin{proof}
For the upper bound, let $G\subseteq X$ be a minimum weakly $K_q$-saturated spanning subgraph.
In $\widehat X$, take as initial edges all edges of $G$ together with all cone edges $av$, $v\in V(X)$.
Then add the missing edges of $X$ in a weak $K_q$-saturation order for $G$.
Each new $K_q$ in $X$ becomes a new $K_{q+1}$ after adjoining $a$.
Thus
\[
    \wsat(\widehat X,K_{q+1})\le \wsat(X,K_q)+|V(X)|.
\]

For the lower bound, let $J\subseteq \widehat X$ be a minimum weakly $K_{q+1}$-saturated spanning subgraph.
Since $\omega(X)\le q$, every copy of $K_{q+1}$ in $\widehat X$ contains $a$.
Let $A_0=N_J(a)$ and $U=V(X)\setminus A_0$.
We call a vertex $v\in V(X)$ \emph{active} if $v\in A_0$ or $av$ has been added.
For each $v\in U$, consider the first time at which $av$ is added.
The new $K_{q+1}$ containing $av$ has the form $\{a\}\cup Q_v$, where $Q_v$ is a $q$-clique in $X$ containing $v$.
Choose a vertex $p(v)\in Q_v\setminus\{v\}$ and set $f_v=vp(v)$.
The edge $f_v$ is present before $av$ is added.
Moreover, $f_v$ must be an edge of $J$: if $f_v$ had been added earlier, then the new $K_{q+1}$ containing $f_v$ would also contain $a$, and hence the edge $av$ would already have been added at that earlier time, a contradiction.
The edges $f_v$, $v\in U$, are pairwise distinct.
Indeed, if $f_u=f_v=uv$ and $au$ is added before $av$, then $v\in Q_u$, so $av$ is already added before $au$ is added, contradicting the assumed order of first appearances.

Let $J_0=J[V(X)]\setminus\{f_v:v\in U\}$ be a subgraph of $X$, then $|E(J_0)|=|E(J)|-|A_0|-|U|=|E(J)|-|V(X)|.$
We show that $J_0$ is weakly $K_q$-saturated in $X$ by projecting the weak saturation process from $\widehat X$ to $X$. 
More precisely, we define the projected process as follows. Each edge in $X$ is added unchanged, while each cone edge $av$ is replaced by the edge $f_v$.
When the projected process adds an edge $e$ of $X$, this corresponds to a step in the original process at which the same edge $e$ is added. At that step, there is a newly created copy of $K_{q+1}$ in $\widehat X$ containing $e$. Since all vertices of this copy that lie in $X$ are active at that moment, deleting the cone vertex $a$ yields a copy of $K_q$ in $X$ containing $e$. When the projected process adds $f_v$, this corresponds to the step in the original process at which the cone edge $av$ is added. At that moment, all vertices of $Q_v$ are active, and hence every edge of $Q_v$ other than $f_v$ has already appeared in the projected process. Therefore adding $f_v$ completes a copy of $K_q$ on $Q_v$ in $X$.
Hence $J_0$ is weakly $K_q$-saturated in $X$.
Therefore
\[
    \wsat(\widehat X,K_{q+1})=|E(J)|=|E(J_0)|+|V(X)|\ge \wsat(X,K_q)+|V(X)|,
\]
which proves the lower bound.
\end{proof}

Now we are ready to prove the Theorem~\ref{thm:intro-cliques} for complete graphs by using the lifting lemma we proved above.
\begin{proof}[\bf Proof of Theorem~\ref{thm:intro-cliques} for case $K_r$]
Let $\varphi$ be a $3$-CNF formula, and let $L_\varphi$ be the finite connected pure $2$-dimensional simplicial complex from Theorem~\ref{thm:TT}. Let $X_\varphi=L_\varphi^{(1)}$ and set $n=|V(X_\varphi)|$.
By Theorem~\ref{thm:TT} and Lemma~\ref{lem:connected-lower}, we have that the following statements are equivalent.
\begin{enumerate}[(1)]
    \item $\varphi$ is satisfiable.
    \item $\wsat(X_\varphi,K_3)=n-1.$
    \item $\wsat(X_\varphi,K_3)\le n-1.$
\end{enumerate}
Since $L_\varphi$ is a $2$-dimensional flag complex, Corollary~\ref{obs:flag-k4free} implies that $\omega(X_\varphi)\le 3$.
For fixed $r\ge 3$, let $F_\varphi=X_\varphi\nabla K_{r-3}$ and set $N=|V(F_\varphi)|=n+r-3$. 
By applying Lemma~\ref{lem:clique-cone} iteratively, note that the clique number increases by exactly one at each join step, we obtain 
$$
    \wsat(F_\varphi,K_r)
    =
    \wsat(X_\varphi,K_3)+(r-3)n+\binom{r-3}{2}.
$$
Hence, $\varphi$ is satisfiable if and only if $\wsat(F_\varphi,K_r)
 \le (n-1)+(r-3)n+\binom{r-3}{2} = (r-2)N-\binom{r-1}{2}.$ This gives a polynomial-time reduction from $3$-SAT to our decision problem. Together with Lemma~\ref{lem:np-membership}, the result follows.
\end{proof}

\section{Balanced complete bipartite graphs}\label{sec:balanced}
In this section, we prove our complexity result for balanced complete bipartite graphs $K_{r,r}$ with $r\ge 3$. We first establish the case $K_{3,3}$. Starting from the problem in Theorem~\ref{thm:TT}, namely deciding, given an $n$-vertex graph $F$, whether $\wsat(F,K_3)=n-1$, we give a polynomial-time reduction to the corresponding decision problem for $K_{3,3}$. The reduction is based on an auxiliary bipartite graph $\dB(G)$ constructed from a graph $G$. We introduce a structural property of graphs, called \emph{$3$-clean}, and prove that, whenever $G$ is $3$-clean,
$$
\wsat(\dB(G),K_{3,3})=\wsat(G,K_3)+|V(G)|+|E(G)|.
$$
To apply this identity to the instances arising from Theorem~\ref{thm:TT}, we show that the simplicial complex $L_\varphi$ appearing there admits a suitable subdivision that is flag-no-square. We then prove that the $1$-skeleton of every flag-no-square simplicial complex is $3$-clean and $K_4$-free. These ingredients connect the topological construction in Theorem~\ref{thm:TT} with our graph-theoretic reduction and yield the desired hardness result for $K_{3,3}$. Finally, we prove a balanced lifting lemma, which extends the result from $K_{3,3}$ to $K_{r,r}$ for every $r\ge 4$.

Now we introduce the definitions of the auxiliary bipartite graph $\dB(G)$ and the $3$-clean property.

\begin{dfn}\label{def:Bgraph}
For a graph $G$, let $V_L$ and $V_R$ be two disjoint copies of $V(G)$. For each vertex $v\in V(G)$, let $v_L\in V_L$ and $v_R\in V_R$ be its corresponding copies respectively. We define $\dB(G)$ to be the bipartite graph with parts $V_L$ and $V_R$ and edge set $E(\dB(G))=\{u_Lv_R: u=v \text{ or } uv\in E(G)\}.$
Equivalently, $\dB(G)$ contains the diagonal edge $v_Lv_R$ for every $v\in V(G)$, and the two cross-edges $u_Lv_R$ and $v_Lu_R$ for every edge $uv\in E(G)$.
\end{dfn}

\begin{dfn}[$3$-clean]\label{def:3clean}
For a graph $G$, let $G^\circ$ be the graph such that $V(G^\circ)=V(G)$, $xy\in E(G^\circ)$ if and only if either $x=y$ or $xy\in E(G)$.
A graph $G$ is \emph{$3$-clean} if, for every $X,Y\in \binom{V(G)}{3}$, $X\times Y\subseteq E(G^\circ)$ implies that $X=Y$ and $G[X]\cong K_3$.
\end{dfn}

We have the following lemma.

\begin{lem}\label{lem:clean-lifting}
Let $G$ be a $3$-clean graph, and set $n=|V(G)|$ and $m=|E(G)|$.
Then
$$
    \wsat(\dB(G),K_{3,3})=n+m+\wsat(G,K_3).
$$
\end{lem}

Before proving this lemma, we define the mixed peeling process. Recall that we defined $H$-edge peeling and $\lambda_H(F)$ in Definition~\ref{def:lambda}.

\begin{dfn}[Mixed peeling]\label{def:mixed-peeling}
Let $H$ be a graph with no isolated vertices, and let $F$ be a host graph. A \emph{mixed $H$-peeling} deletes, at each step, either one edge of a current copy of $H$, or one vertex of such a copy together with all currently incident edges.
Let $\mu_H(F)$ be the maximum length of a mixed $H$-peeling sequence.
\end{dfn}

It is easy to see that the mixed peeling is dominated by the edge peeling.

\begin{obs}\label{lem:mixed-peeling}
If $H$ has no isolated vertices, then for every finite graph $F$, we have
$$
    \mu_H(F)= \lambda_H(F).
$$
\end{obs}

\begin{proof}
Since a $H$-edge peeling is a mixed $H$-peeling, $\mu_H(F)\ge \lambda_H(F)$. Now we need to prove that $\mu_H(F)\le \lambda_H(F)$.
We prove this by induction on $|V(F)|+|E(F)|$.
If $F$ is $H$-free, there is nothing to prove.

For the case when $F$ is not $H$-free, let us first take an optimal mixed $H$-peeling sequence and inspect its first step. If the first step deletes an edge $e$, then consider an $H$-edge peeling which deletes $e$.
The induction hypothesis applied to $F-e$ gives
$$
    \mu_H(F)=1+\mu_H(F-e)
    \le 1+\lambda_H(F-e)
    \le \lambda_H(F).
$$

If the first step deletes a vertex $v$, then $v$ lies in a current copy of $H$.
Since $H$ has no isolated vertices, $v$ is incident with some edge $e$ of $H$.
Consider an $H$-edge peeling which delete $e$.
The graph $F-v$ is a subgraph of $F-e$, so by induction and monotonicity of $\lambda_H$ in the host graph, we have $\mu_H(F-v)\le \lambda_H(F-v)\le \lambda_H(F-e).$
Thus
$$
    \mu_H(F)=1+\mu_H(F-v)
    \le 1+\lambda_H(F-e)
    \le \lambda_H(F).
$$
\end{proof}

\begin{proof}[Proof of Lemma~\ref{lem:clean-lifting}]
Let $G$ be a $3$-clean graph, and set $n=|V(G)|$ and $m=|E(G)|$. For every copy of $K_{3,3}$ in $\dB(G)$,
let $X,Y\in\binom{V(G)}{3}$ be its two parts, then $X\times Y\subseteq E(G^\circ).$
Since $G$ is $3$-clean, this implies $X=Y$ and $G[X]\cong K_3$.
Thus every $K_{3,3}$ in $\dB(G)$ comes from a unique triangle of $G$.

We next prove that $\lambda_{K_{3,3}}(\dB(G))=\lambda_{K_3}(G)$. Fix an arbitrary orientation of each edge of $G$; we refer to this orientation as the \emph{chosen orientation} of the edge.
The opposite oriented edge in $\dB(G)$, together with all diagonal edges, will be used as anchors and never deleted.

For the lower bound, whenever a $K_3$-edge peeling sequence in $G$ deletes an edge $uv$ contained in a current triangle $uvw$, delete in $\dB(G)$ the edge corresponding to the chosen orientation of $uv$.
The triangle $uvw$ gives a $K_{3,3}$ in $\dB(G)$, and all other eight edges of this $K_{3,3}$ are still present because the other edges in the triangle have not yet been deleted in $G$ and all anchors remain.
Hence $\lambda_{K_{3,3}}(\dB(G))\ge \lambda_{K_3}(G)$.

For the upper bound, take the optimal $K_{3,3}$-edge peeling sequence in $\dB(G)$, we construct a mixed $K_3$-peeling sequence $G$ as follows. Whenever the $K_{3,3}$-edge peeling deletes a non-diagonal edge $u_Lv_R$ with $u\neq v$, the mixed $K_3$-peeling deletes the edge $uv$ in $G$.
Whenever the $K_{3,3}$-edge peeling deletes a diagonal edge $v_Lv_R$, the mixed $K_3$-peeling deletes the vertex $v$ together with all edges currently incident to $v$. Since every copy of $K_{3,3}$ in $\dB(G)$ arises from a triangle in $G$, each step above is legal. Hence this construction gives a valid mixed $K_3$-peeling sequence in $G$.
Therefore $\lambda_{K_{3,3}}(\dB(G))\le \mu_{K_3}(G)= \lambda_{K_3}(G)$ by Lemma~\ref{lem:mixed-peeling}.

Finally, since $|E(\dB(G))|=n+2m$,
using Lemma~\ref{lem:reverse-deletion}, we have $\wsat(G,K_3)=m-\lambda_{K_3}(G)$ and 
\begin{align*}
    \wsat(\dB(G),K_{3,3})
    &=|E(\dB(G))|-\lambda_{K_{3,3}}(\dB(G))\\
    &=n+2m-\lambda_{K_3}(G)
    =n+m+\wsat(G,K_3).
\end{align*}
\end{proof}

\subsection{Flag-no-square and special subdivision}

Now we introduce the \emph{flag-no-square} property and the \emph{special subdivision}. 
\begin{dfn}[Flag-no-square]
A $2$-dimensional complex $L$ is \emph{no-square} if $L^{(1)}$ has no induced $4$-cycle.
It is \emph{flag-no-square} if it is both flag and no-square.
\end{dfn}

The special subdivision is from a construction due to Dranishnikov~\cite{Dra99}, we recall it there. Let
$\mathcal I$ be the boundary complex of the icosahedron, and let $\tau$ be
a fixed $2$-simplex of $\mathcal I$. Define $Z_{10}$ to be the
subcomplex of $\mathcal I$ spanned by all vertices whose distance from
$\tau$ in the $1$-skeleton $\mathcal I^{(1)}$ is at most one. Under this definition, the complex $Z_{10}$ is a triangulated $2$-dimensional disc consisting of ten
$2$-simplices. Among its vertices, exactly three are incident with precisely
two $2$-simplices of $Z_{10}$; these vertices lie on the boundary of
$Z_{10}$ and are called the \emph{true vertices}.

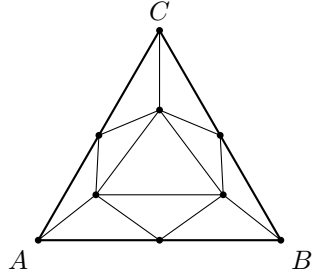
\begin{figure}[htbp]
\centering
\begin{tikzpicture}[scale=0.8, every node/.style={font=\small}]

\coordinate (A) at (0,0);
\coordinate (B) at (4,0);
\coordinate (C) at (2,3.464);

\coordinate (x) at (2,0);
\coordinate (y) at (3,1.732);
\coordinate (z) at (1,1.732);

\coordinate (p) at (0.95,0.75);
\coordinate (q) at (3.05,0.75);
\coordinate (r) at (2,2.15);

\draw[thick] (A)--(x)--(B)--(y)--(C)--(z)--cycle;

\draw (A)--(p);
\draw (x)--(p);
\draw (x)--(q);
\draw (B)--(q);
\draw (y)--(q);
\draw (y)--(r);
\draw (C)--(r);
\draw (z)--(r);
\draw (z)--(p);
\draw (p)--(q)--(r)--cycle;

\foreach \v in {A,B,C,x,y,z,p,q,r}
  \fill (\v) circle (1.6pt);

\node[below left] at (A) {$A$};
\node[below right] at (B) {$B$};
\node[above] at (C) {$C$};

\end{tikzpicture}
\caption{The picture of $Z_{10}$, where $A,B,C$ are the true vertices.}
\label{fig:Z10}
\end{figure}

The following definition of the special subdivision is the Definition~2.2 in~\cite{PS09}.
\begin{dfn}
The \emph{special subdivision} of a $2$-simplex $\Delta$ is the subdivision isomorphic to $Z_{10}$, where vertices of $\Delta$ correspond to true vertices of $Z_{10}$. We denote it by $\Delta^*$. For any $2$-dimensional simplicial complex $Y$, the \emph{special subdivision} $Y^*$ is obtained by taking the first barycentric subdivision of the $1$-skeleton of $Y$, followed by the special subdivision of every $2$-simplex of $Y$.
\end{dfn}

Przytycki and {\'S}wi{\k a}tkowski~\cite{PS09} also proved the following lemma, which appears as the Lemma~2.3 in their paper.
\begin{lem}[Przytycki and {\'S}wi{\k a}tkowski~\cite{PS09}]\label{lem:special-fns}
    Let $Y$ be a $2$-dimensional simplicial complex. Then its special subdivision $Y^*$ satisfies the flag-no-square property.
\end{lem}

Next, we prove a lemma to relate the flag-no-square property to our $3$-clean property.
\begin{lem}\label{lem:fns-clean}
If $L$ is a $2$-dimensional flag-no-square complex and $G=L^{(1)}$, then $G$ is $3$-clean and $K_4$-free.
\end{lem}

\begin{proof}
The graph $G$ is $K_4$-free by Corollary~\ref{obs:flag-k4free}, and it has no induced $C_4$ by the no-square property. We first claim that if $a$ and $b$ are nonadjacent vertices of $G$, then they have at most two common neighbours.
Otherwise, let $u,v,w$ be three common neighbours of $a$ and $b$.
If two of $u,v,w$, say $u$ and $v$, are nonadjacent, then $a-u-b-v-a$ is an induced $C_4$.
If $u,v,w$ are pairwise adjacent, then $\{a,u,v,w\}$ spans a $K_4$.
Both cases lead to a contradiction.

Now suppose $X,Y\in\binom{V(G)}{3}$ and $X\times Y\subseteq E(G^\circ)$.
If $|X\cap Y|=0$, then no two vertices of $X$ can be nonadjacent, since they would have the three vertices of $Y$ as common neighbours.
Thus $G[X]$ is a $K_3$, then any vertex of $Y$ together with $X$ gives a $K_4$ in $G$, a contradiction.

If $|X\cap Y|=1$, write $X=\{z,a,b\}$ and $Y=\{z,c,d\}$.
The vertices $a$ and $b$ must be adjacent, since otherwise $z,c,d$ would be three common neighbours of $a$ and $b$, leads to a contradiction. Thus $\{z, a, b, c\}$ spans a $K_4$ in $G$, a contradiction.

If $|X\cap Y|=2$, write $X=\{p,q,a\}$ and $Y=\{p,q,b\}$.
The condition $X\times Y\subseteq E(G^\circ)$ gives $pq,pa,pb,qa,qb,ab\in E(G),$ so $\{p,q,a,b\}$ spans a $K_4$ in $G$, which is a contradiction.

Therefore $X=Y$.
Then $X\times X\subseteq E(G^\circ)$ forces the three vertices of $X$ to be pairwise adjacent, so $G[X]\cong K_3$. Thus $G$ is $3$-clean and $K_4$-free.
\end{proof}

\subsection{The reduction in the \texorpdfstring{$K_{3,3}$}{} case}

Recall that we aim to perform the special subdivision on the simplicial complex $L_\varphi$ in Theorem~\ref{thm:TT}. The next two lemmas tell us that the weak saturation number is related to the collapsibility of some simplicial complex and the special subdivision keeps the collapsibility.

\begin{lem}\label{lem:flag-collapse-wsat}
Let $L$ be a finite connected $2$-dimensional flag complex, let $G=L^{(1)}$, and set $n=|V(L)|$.
Then
$$
    \wsat(G,K_3)=n-1
$$
if and only if $L$ becomes collapsible after deleting some $2$-dimensional faces.
\end{lem}

\begin{proof}
Suppose first that $\wsat(G,K_3)=n-1$.
Let $T\subseteq G$ be a weakly $K_3$-saturated spanning subgraph with $n-1$ edges.
By Lemma~\ref{lem:connected-lower}, $T$ is connected, and hence it is a spanning tree.

Let $e_1,\ldots,e_m$ be a weak saturation order for the edges of $G$ outside $T$.
When $e_i$ is added, it creates a new triangle.
Since $L$ is a flag, this new triangle containing $e_i$ in $G$ derives a $2$-dimensional face of $L$, denote it by $\theta_i$.
The faces $\theta_i$ are pairwise distinct.

Let $L'$ be the subcomplex obtained from $L$ by deleting all $2$-dimensional faces not among $\theta_1,\ldots,\theta_m$.
We collapse $L'$ by removing $\{e_i,\theta_i\}$ with order $\{e_m,\theta_m\},\{e_{m-1},\theta_{m-1}\},\ldots,\{e_1,\theta_1\}.$
At the moment when $\{e_i,\theta_i\}$ is considered, no remaining $2$-dimensional face other than $\theta_i$ contains $e_i$.
Indeed, For $j<i$, the face $\theta_j$ does not contain $e_i$, since $e_i$ had not yet appeared at step $j$ in the weak saturation order; for $j>i$, the face $\theta_j$ has already been removed in the collapse sequence.
Thus $e_i$ is a free face of $\theta_i$.
After these collapses, only the tree $T$ remains, and it is easy to see that $T$ is collapsible.
Hence $L'$ is collapsible.

Conversely, suppose that a subcomplex $L'$ obtained from $L$ by deleting some $2$-dimensional faces is collapsible.
Choosing an elementary collapse sequence from $L'$ to a single vertex.
Without loss of generality, We may assume that all elementary collapses in which the facet $\sigma$ has dimension $2$ are performed before those in which $\sigma$ has dimension $1$.
Indeed, an elementary collapse whose facet is an edge removes only a maximal edge and one of its free endpoints, and hence removes no simplex contained in a remaining $2$-dimensional face. Thus postponing such a collapse does not change the set of remaining $2$-dimensional faces, nor which edges are contained in them. Therefore any subsequent elementary collapse whose facet is $2$-dimensional is still valid if performed first. The postponed edge collapse also remains valid afterwards, since deleting simplices cannot make its free endpoint belong to an additional facet. Repeating this interchange gives the desired ordering.
After all $2$-dimensional faces have been removed, the remaining $1$-dimensional complex forms a connected graph $T$ on all vertices.
Since $L'$ is collapsible, by Fact~\ref{lem:collapse-euler}, we have $\chi(T)=\chi(L')=1$. Hence $T$ is a tree.

The reverse order of the elementary collapses whose facet is $2$-dimensional gives an edge-addition order from $T$ to $G$ in which each added edge creates a triangle.
Thus $T$ is weakly $K_3$-saturated in $G$, so $\wsat(G,K_3)\le n-1$. By Lemma~\ref{lem:connected-lower}, we have $\wsat(G,K_3)\ge n-1$, and therefore we have $\wsat(G,K_3)= n-1$.
\end{proof}

\begin{lem}\label{lem:subdivision-reflection}
Let $L$ be a finite connected $2$-dimensional simplicial complex, and let $L'$ be a subdivision of $L$.
Then $L$ becomes collapsible after removing some $2$-dimensional faces if and only if $L'$ becomes collapsible after removing some $2$-dimensional faces.
\end{lem}

\begin{proof}
Suppose first that $L\setminus H$ is collapsible for some set $H\subseteq L^{(2)}$ of $2$-dimensional faces.
For each $\Delta\in H$, let $D_\Delta$ be the triangulated disc in $L'$ which is a subdivision of $\Delta$.
By Fact~\ref{lem:disk-endo-collapsible}, we can choose a small triangle $\sigma_\Delta\subseteq D_\Delta$ such that $(D_\Delta \setminus \sigma_\Delta)\searrow \partial D_\Delta.$
Let $P=\{\sigma_\Delta:\Delta\in H\}$.
In $L'-P$, collapse each $D_\Delta \setminus \sigma_\Delta$ to its boundary implies $(L' \setminus P)\searrow (L \setminus H)',$ where $(L \setminus H)'$ is the subdivision of $L \setminus H$ induced by $L'$.
By Fact~\ref{lem:subdivision-collapsible}, $(L \setminus H)'$ is collapsible, so $L' \setminus P$ is collapsible.

Conversely, suppose that $L' \setminus S$ is collapsible for some set $S$ of $2$-dimensional faces in $L'$.
Let $R$ be the set of $2$-dimensional faces in $L$ whose subdivision contains at least one $2$-dimensional face in $S$.
Call an $2$-dimensional face of $L$ \emph{intact} if it is not in $R$.
Fix an elementary collapse sequence of $L' \setminus S$.
For every intact $2$-dimensional face $\Delta$, let $\tau(\Delta)$ be the first time at which a small $2$-dimensional face contained in the subdivision of $\Delta$ is removed in this sequence. We order the intact $2$-dimensional faces by increasing $\tau$.

Consider an intact $2$-dimensional face $\Delta$ whose time $\tau(\Delta)$ is currently the smallest, and let $\delta$ be the first small $2$-dimensional face in the subdivision of $\Delta$ that will be removed by the collapse sequence.
At that moment $\delta$ has a free edge $f$.
Since no small $2$-dimensional face in the subdivision of $\Delta$ has been removed earlier, $f$ cannot be an interior edge of that subdivided disk. Thus $f$ lies on the subdivision of an edge $e\subseteq \partial\Delta$.

We now claim that $e$ is a free edge in the current complex obtained from $L \setminus R$. Indeed, if there is another intact $2$-dimensional face $\Delta_1$ contains $e$, then the minimality of $\tau(\Delta)$ implies that the small $2$-dimensional face in the subdivision of $\Delta_1$ that contains $f$ would still be present when $\delta$ is removed, contradicting that $f$ is free. We may therefore collapse $\Delta$ along the free edge $e$.

Repeating this for the intact $2$-dimensional faces, in the order given by $\tau$, collapses $L \setminus R$ to a $1$-dimensional connected graph $T$.
The connectivity is preserved because removing $2$-dimensional faces does not change the $1$-skeleton, and collapsing a $2$-dimensional face along a free edge does not disconnect the $1$-skeleton: the endpoints of the removed edge remain connected through the other two edges.

It remains to see that $T$ is a tree. Let $s=|S|$ and $r=|R|$.
Since $L' \setminus S$ is collapsible, we have $\chi(L' \setminus S)=1$. It is well known that subdivision preserves Euler characteristic, and removing one $2$-dimensional face decreases Euler characteristic by $1$, so $\chi(L)-s=1.$
Also, we have 
$$
    \chi(T)=\chi(L \setminus R)=\chi(L)-r=1+s-r.
$$
By definition, every $2$-dimensional face in $R$ contains at least one $2$-dimensional face in $S$, so $r\le s$ and hence $\chi(T)\ge 1$.
Since a connected graph has Euler characteristic at most $1$, we have $\chi(T)=1$ and $T$ is a tree.
Thus $T$ is collapsible, which implies $L \setminus R$ is collapsible.
\end{proof}

Now we are ready to prove that given a finite graph $F$ and an integer $k$ as input, it is NP-complete to decide whether $\wsat(F,K_{3,3})\le k$.

\begin{proof}[\bf Proof of Theorem~\ref{thm:intro-cliques} for case $K_{3,3}$]
    For any $3$-CNF formula $\varphi$, let $L_\varphi$ be the complex from Theorem~\ref{thm:TT}, and let $L_\varphi^*$ be the special subdivision of $L_\varphi$. By Lemma~\ref{lem:special-fns} and Lemma~\ref{lem:fns-clean}, we have $L_\varphi^*$ is flag-no-square and $(L_\varphi^*)^{(1)}$ is $3$-clean. We claim that the following statements are equivalent.
    \begin{enumerate}[(1)]
        \item $\varphi$ is satisfiable.
        \item $L_\varphi$ becomes collapsible after removing a set of 2-dimensional faces.
        \item $L_\varphi^*$ becomes collapsible after removing a set of 2-dimensional faces.
        \item $\wsat\left((L_\varphi^*)^{(1)},K_3\right)=|V\left( (L_\varphi^*)^{(1)}\right) |-1.$
    \end{enumerate}
    The equivalence between (1) and (2) is from Theorem~\ref{thm:TT}; Lemma~\ref{lem:subdivision-reflection} gives the equivalence between (2) and (3) and Lemma~\ref{lem:flag-collapse-wsat} gives the equivalence between (3) and (4).

    Set $n=|V\left( (L_\varphi^*)^{(1)}\right) |$ and $m=|E\left( (L_\varphi^*)^{(1)}\right) |$. By Lemma~\ref{lem:clean-lifting}, we have $\wsat\left(\dB\left((L_\varphi^*)^{(1)}\right),K_{3,3}\right) = n+m+\wsat\left((L_\varphi^*)^{(1)},K_3\right).$ Since $(L_\varphi^*)^{(1)}$ is connected, Lemma~\ref{lem:connected-lower} gives $\wsat\left((L_\varphi^*)^{(1)},K_3\right)\ge n-1$. Together with the analysis above, we have that $\varphi$ is satisfiable if and only if $\wsat\left((L_\varphi^*)^{(1)},K_3\right)\le n-1$, which is equivalent to $\wsat\left(\dB\left((L_\varphi^*)^{(1)}\right),K_{3,3}\right) \le 2n+m-1.$ 
    This gives a polynomial-time reduction from $3$-SAT to our decision problem. Together with Lemma~\ref{lem:np-membership}, the result follows.
\end{proof}

\subsection{Lifting to \texorpdfstring{$K_{r,r}$}{} with \texorpdfstring{$r \ge 4$}{}}

For the general case, given a $3$-CNF formula $\varphi$, we construct a bipartite graph from the graph $\dB\left((L_\varphi^*)^{(1)}\right)$ used in the proof of the $K_{3,3}$ case, and reduce $3$-SAT to the corresponding decision problem. More precisely, if $(V_L,V_R)$ is the bipartition of $\dB\left((L_\varphi^*)^{(1)}\right)$, we add $r-3$ new vertices to $V_L$ and $r-3$ new vertices to $V_R$, and then join every new vertex completely to the opposite part.

Before giving the reduction, we prove several auxiliary lemmas. Let $X$ be a bipartite graph with left side $A$ and right side $B$.
We define a \emph{lifting} on $X$ as follows: add a new vertex $u$ to $A$ and a new vertex $v$ to $B$, join $u$ to every vertex of $B$, join $v$ to every vertex of $A$, and add the edge $uv$. We denote the resulting bipartite graph by $X^+$. It is easy to see that for every integer $a>0$, if $X$ is $K_{a+1,a}$-free, then every copy of $K_{a+1,a+1}$ in $X^+$ contains both $u$ and $v$.

\begin{lem}\label{lem:sync-lifting}
For every integer $a>0$, if a bipartite graph $X$ with parts $A$ and $B$ is $K_{a+1,a}$-free, then $$\wsat(X^+,K_{a+1,a+1})=\wsat(X,K_{a,a})+|A|+|B|+1.$$
\end{lem}

\begin{proof}
By Lemma~\ref{lem:reverse-deletion}, it is enough to prove
$\lambda_{K_{a+1,a+1}}(X^+)=\lambda_{K_{a,a}}(X).$ Let $u$ and $v$ be the new vertices in $X^+$ such that the parts of $X^+$ is $A\cup \{u\}$ and $B\cup \{v\}$. Since $X$ is $K_{a+1,a}$-free, by the observation above, every copy of $K_{a+1,a+1}$ in $X^+$ contains both $u$ and $v$. So there is a bijection between the copies of $K_{a,a}$ in $X$ and the copies of $K_{a+1,a+1}$ in $X^+$. 

The lower bound is immediate. Given a $K_{a,a}$-edge peeling sequence in $X$, it naturally forms a $K_{a+1,a+1}$-edge peeling sequence in $X^+$.

To prove the upper bound, we take an arbitrary $K_{a+1,a+1}$-edge peeling sequence $\mathcal P$ in $X^+$ and show how to derive from it a mixed $K_{a,a}$-peeling sequence $\mathcal P'$ in $X$.
We define $\mathcal P'$ as follows. Whenever $\mathcal P$ deletes an edge $xy\in E(X)$, we delete the same edge $xy$ in $\mathcal P'$. Whenever $\mathcal P$ deletes an edge $u y$ with $y\in B$, we delete the vertex $y$ in $\mathcal P'$, together with all edges of $X$ currently incident to $y$. Similarly, whenever $\mathcal P$ deletes an edge $xv$ with $x\in A$, we delete the vertex $x$ in $\mathcal P'$, together with all edges of $X$ currently incident to $x$.

It remains to handle the edge $uv$. Suppose that $uv$ is deleted at some step of $\mathcal P$. Then this step must be the last step of the peeling sequence: any copy of $K_{a+1,a+1}$ involved in a later deletion would have to contain both $u$ and $v$, and hence would require the edge $uv$. For this final step, we let $\mathcal P'$ delete an arbitrary edge contained in the corresponding copy of $K_{a,a}$ in $X$.

The mixed $K_{a,a}$-peeling sequence $\mathcal P'$ is legal. Indeed, at each step of $\mathcal P'$, the deleted edge or vertex is contained in a copy of $K_{a,a}$ in the current graph, obtained from the copy of $K_{a+1,a+1}$ that contains the edge deleted at the corresponding step of $\mathcal P$ and is destroyed by that deletion. Moreover, no edge or vertex is deleted more than once in $\mathcal P'$.

Thus we have
$$
    \lambda_{K_{a+1,a+1}}(X^+)
    \le \mu_{K_{a,a}}(X)
    = \lambda_{K_{a,a}}(X),
$$
where the last equality is by Lemma~\ref{lem:mixed-peeling}.
And hence we have $\lambda_{K_{a+1,a+1}}(X^+)=\lambda_{K_{a,a}}(X).$ This completes the proof of Lemma~\ref{lem:sync-lifting}.
\end{proof}

For $3$-clean graphs, we also have the following lemma.
\begin{lem}\label{lem:B-no-43}
If $G$ is $3$-clean, then $\dB(G)$ is $K_{3,4}$-free.
\end{lem}

\begin{proof}
Suppose that $\dB(G)$ contains a copy of $K_{3,4}$. Let $X$ and $Y$ be its two parts. Without loss of generality, we assume that $|X|=3$ and $|Y|=4$. We identify $X$ and $Y$ with their corresponding subsets of $V(G)$. For every $Y'\subset Y$ with $|Y'|=3$, the definition of $\dB(G)$ implies that $X\times Y'\subseteq E(G^\circ)$. Since $G$ is $3$-clean, it follows that $X=Y'$. This would have to hold for every $3$-element subset $Y'$ of $Y$, which is impossible because $|Y|=4$. Hence $\dB(G)$ is $K_{3,4}$-free.
\end{proof}

Now we are ready to prove the $K_{r,r}$ case with $r\ge 4$.

\begin{proof}[\bf Proof of Theorem~\ref{thm:intro-cliques} for case $K_{r,r}$]
    For any $3$-CNF formula $\varphi$, let $L_\varphi$ be the complex from Theorem~\ref{thm:TT}, and let $L_\varphi^*$ be the special subdivision of $L_\varphi$. For $r\ge 4$, we define $\dB_{\varphi}^r$ to be the graph obtained from $\dB\left((L_\varphi^*)^{(1)}\right)$ by applying the lifting operation successively $r-3$ times. By Lemma~\ref{lem:special-fns} and Lemma~\ref{lem:fns-clean}, we have $(L_\varphi^*)^{(1)}$ is $3$-clean. Then Lemma~\ref{lem:B-no-43} implies that $\dB\left((L_\varphi^*)^{(1)}\right)$ is $K_{3,4}$-free. Hence, by the definition of the lifting operation, it is easy to see that $\dB_{\varphi}^r$ is $K_{r,r+1}$-free for every $r\ge 4$.

    As before, set $n=|V\left( (L_\varphi^*)^{(1)}\right) |$ and $m=|E\left( (L_\varphi^*)^{(1)}\right) |$.
    Therefore, by repeated application of Lemma~\ref{lem:sync-lifting}, we obtain
\begin{align*}
    \wsat\left(\dB_{\varphi}^r,K_{r,r}\right)
    &=
    \wsat\left(\dB\left((L_\varphi^*)^{(1)}\right),K_{3,3}\right)
    +
    \sum_{i=3}^{r-1}\bigl(2(n+i-3)+1\bigr)\\
    &=
    \wsat\left(\dB\left((L_\varphi^*)^{(1)}\right),K_{3,3}\right)+2(r-3)n+(r-3)^2.
\end{align*}
   By the proof of the $K_{3,3}$ case, we obtain that $\varphi$ is satisfiable if and only if $\wsat\left(\dB_{\varphi}^r,K_{r,r}\right) \le 2(r-2)n+m+(r-3)^2-1$. This gives a polynomial-time reduction from $3$-SAT to our decision problem. Together with Lemma~\ref{lem:np-membership}, it completes the proof of Theorem~\ref{thm:intro-cliques}.
\end{proof}

\section{Concluding Remarks}\label{sec:concluding}
We determine the NP-completeness of the weak saturation threshold problem for all cliques $K_r$ ($r\ge 3$) and balanced complete bipartite graphs $K_{r,r}$ ($r\ge 3$). The proof builds on the triangle case of Tancer and Tyomkyn~\cite{TT25}. For cliques, we apply simple lifting operations by adding a universal vertex at each step. For bipartite graphs, we use a construction that converts triangles into copies of $K_{3,3}$. This construction relies on the $3$-clean property, which we introduce, and on topological tools such as special subdivision.

The method in this paper does not extend to $C_4=K_{2,2}$ or to general $K_{s, t}$. In the case of $C_4$, since $\wsat(F,K_2)=0$, there is no meaningful analogue of a $2$-clean reduction. In the general case, a copy of $K_{s,t}$ in a bipartite host graph does not come with a prescribed assignment of its two parts to the two sides of the host graph: either side may realize the $s$-part or the $t$-part of the $K_{s,t}$. This causes no difficulty in the balanced case $K_{r,r}$, but when $s\ne t$ it leads to additional configurations that must be excluded.
For example, attempting to pass from $K_{3,3}$ to $K_{3,4}$ by padding only one side would require excluding structure such as $K_{4,2}$ in the base graph. The $3$-clean property rules out the copies $K_{a,a+1}$ needed for the balanced lifting, but it does not control such $K_{a,a+c}$ when $c\ge 2$. New hard families or a different forcing mechanism would therefore be needed to handle general $K_{s,t}$.

\section*{Acknowledgement}
Tianying Xie was supported by National Key R and D Program of China  2023YFA1010201 and National Natural Science Foundation of China grants 12501474 and 12471336.

\section*{Declaration on the use of AI}
The authors used generative AI tools to assist in discussing proof strategies, checking proofs, and improving the exposition. The authors take full responsibility for the mathematical arguments, results, and conclusions, all of which were carefully reviewed and verified by them.

\bibliographystyle{unsrt}

\end{document}